\documentclass[12pt]{amsart}
\usepackage{amsfonts, amsmath, amsthm, amssymb}
\usepackage{url, fancyhdr}
\usepackage[all]{xy}

\input xy
\xyoption{all}

\newtheorem{thm}{Theorem}[section]
\newtheorem{lem}[thm]{Lemma}
\newtheorem{cor}[thm]{Corollary}

\newtheorem{prop}[thm]{Proposition}

\theoremstyle{definition}
\newtheorem{example}[thm]{Example}

\newtheorem{defn}[thm]{Definition}

\newenvironment{axiom}{\begin{list}{$\bullet$}{\setlength{\labelsep}{.7cm}
\setlength{\leftmargin}{2.5cm}\setlength{\rightmargin}{0cm}%
\setlength{\labelwidth}{1.8cm}\setlength{\itemsep}{0pt}}}{\end{list}}
\newcommand{\ax}[1]{\item[{\bf #1}\hfill]\index{#1}}

\def\bN{\mathbb{N}}
\def\bQ{\mathbb{Q}}
\def\bR{\mathbb{R}}

\begin{document}

\title[Partial $n$-metric spaces and fixed point theorems]{Partial $n$-metric spaces and fixed point theorems}
\author{Samer Assaf}
\address{Department of Mathematics and Statistics,
University of Saskatchewan, 106 Wiggins Road,
Saskatoon, Saskatchewan, Canada S7N 5E6}
\email{ska680@mail.usask.ca}
\author{Koushik Pal}
\address{Department of Mathematics and Statistics,
University of Saskatchewan, 106 Wiggins Road,
Saskatoon, Saskatchewan, Canada S7N 5E6}
\email{koushik.pal@usask.ca}
\thanks{The second-named author is partially supported by a PIMS Postdoctoral Fellowship.}


\begin{abstract}\noindent
In this paper we combine the notions of partial metric spaces with negative distances, $G_p$-metric spaces and $n$-metric spaces together into one structure called the partial $n$-metric spaces. These are generalizations of all the said structures, and also generalize the notions of $G$-metric and $G_p$-metric spaces to arbitrary finite dimension. We prove Cauchy mapping theorems and other fixed point theorems for such spaces.
\end{abstract}
\maketitle

\section{Introduction}
G{\"a}hler \cite{Gahler1, Gahler2} introduced the notion of 2-metric spaces as a possible generalization of metric spaces. The 2-metric $d(x, y, z)$ is a function of 3 variables, and was intended by G{\"a}hler to be geometrically interpreted as the area of a triangle with vertices at $x$, $y$ and $z$ respectively. However, as several authors (for example, \cite{KSHYJCAW}) pointed out, G{\"a}hler's construction is not a generalization of, rather is independent of, metric spaces. There are results that hold in one but not the other. 

This led B. C. Dhage, in his PhD thesis in 1992, to introduce the notion of a $D$-metric \cite{BDAPBR} that does, in fact, generalize metric spaces. Geometrically, $D(x, y, z)$ can be interpreted as the perimeter of a triangle with vertices at $x$, $y$ and $z$. Subsequently, Dhage published a series of papers attempting to develop topological structures in such spaces and prove several fixed point results. 

In 2003, Mustafa and Sims \cite{ZMBS_rem} demonstrated that most of the claims concerning the fundamental topological properties of $D$-metric spaces are incorrect. This led them to introduce the notion of a $G$-metric \cite{ZMBS_GM}. The interpretation of the perimeter of a triangle applies to a $G$-metric too. Since then, many authors have obtained fixed point results for $G$-metric spaces.

In an attempt to generalize the notion of a $G$-metric space to more than three variables, Khan first introduced the notion of a $K$-metric, and later the notion of a generalized $n$-metric space (for any $n\ge 2$) \cite{KAK_nTop, KAK_nmetric}. He also proved a common fixed point theorem for such spaces.

In a completely different direction Steve Matthews, in his PhD thesis in 1992, introduced the notion of a partial metric space $(X, p)$ \cite{SM1, SM2}, which is also a generalization of a metric space with the essential difference that a point $x\in X$ is allowed to have a nonzero self-distance, i.e., $p(x, x)$ can be nonzero. Matthews proved a Contraction Mapping Theorem for such spaces that generalizes Banach's Fixed Point Theorem for metric spaces. S. J. O'Neill, in his PhD thesis in 1996, generalized Matthews' definition of a partial metric to allow for negative distances \cite{SJON}. In a recent work \cite{SAKP_PM}, the authors have improved on Matthews' Contraction Mapping Theorem in several ways, one of which is to make it work for partial metric spaces with negative distances in the sense of O'Neill.

To combine the two notions of a $G$-metric and a partial metric, Zand and Nezhad introduced the notion of a $G_p$-metric\cite{AZDN} in 2011. They worked with partial metric spaces in the sense of Matthews and proved a fixed point theorem for such spaces.

In this paper, we combine all these concepts together. We work with partial metric spaces in the sense of O'Neill, i.e., we allow for negative distances, and combine them with $G$-metric and generalized $n$-metric spaces to get a structure that generalizes $G_p$-metric to arbitrary finite dimension. We call such structures ``partial $n$-metric spaces''. We also present a new definition of an $n$-metric space that generalizes Khan's definition of a generalized $n$-metric space. Parallel to our earlier work on partial metric spaces with negative values (cf. \cite{SAKP_PM}), we also prove Cauchy Mapping Theorems and other fixed point theorems for partial $n$-metric spaces in this paper. We should mention here that in our fixed point theorems we always restrict ourselves to the orbits of the concerned function and use conditions weaker than completeness, for example orbital completeness (cf. Definition~\ref{orbcomp}), to get a fixed point. Also, because we work with orbits, we potentially have a fixed point for each orbit of the function, which essentially means we loose the uniqueness of a fixed point.

\section{Definitions}
{\bf Notation:} We denote a sequence $\langle x_1, \ldots, x_k\rangle$ by $\langle x_i\rangle_{i=1}^k$. We denote a constant sequence $\langle x, x, \ldots, x\rangle$ of length $k$ by $\langle x\rangle_{i=1}^k$, or simply by $\langle x\rangle_1^k$ (if the index $i$ is understood from the context).\newline

For the rest of this paper, we fix $n\in\bN$ with $n\ge 2$. The following is the fundamental definition of this paper.
\begin{defn}
A pair $(X, G)$ is called a {\em partial $n$-metric space} if $X$ is a nonempty set and $G:X^n \to \bR$ is a function (called the {\em partial $n$-metric}) that satisfies the following four conditions for all $n$-tuples $\langle x_1, \ldots, x_n\rangle\in X^n$ and for all $x, y\in X$:
\begin{axiom}
\ax{(sep)} $G(\langle x \rangle_1^{n-1}, y) = G(\langle x\rangle_1^n) \wedge  G(\langle y \rangle_1^{n-1}, x) = G(\langle y\rangle_1^n) \iff x = y$,
\ax{(ssd)} $G(\langle x\rangle_1^n) \le G(\langle x\rangle_1^{n-1}, y)$, 
\ax{(sym)} $G(\langle x_i\rangle_{i=1}^n) = G(\langle x_{\pi(i)}\rangle_{i=1}^n)\;\;\;\;\;\mbox{ where $\pi$ is a permutation of } \{1, \ldots, n\}$,
\ax{(ptri)} $G(\langle x_i\rangle_{i=1}^n) \le G(\langle x_i\rangle_{i=1}^{n-1}, y) + G(\langle y\rangle_1^{n-1}, x_n) - G(\langle y\rangle_1^n)$.
\end{axiom}
We say $(X, G)$ is {\em strong} if conditions {\bf(sep)} and {\bf(ssd)} are replaced by the condition
\begin{axiom}
\ax{(sssd)} $G(\langle x\rangle_1^n) < G(\langle x\rangle_1^{n-1}, y)$ for all $x, y\in X$ with $x\neq y$. 
\end{axiom}
And we say {\em the self distances of $(X, G)$ are bounded below by $r\in\bR$} if for all $x\in X$
$$r \le G(\langle x\rangle_1^n).$$
\end{defn}

It should be noted that even though we don't mention the ``{\em separation}'' condition {\bf(sep)} explicitly for strong partial $n$-metric spaces, it follows from condition {\bf(sssd)} : if $x \not=y$, then $G(\langle x\rangle_1^n) < G(\langle x\rangle_1^{n-1}, y)$; therefore, $G(\langle x\rangle_1^n) = G(\langle x\rangle_1^{n-1}, y)\implies x = y$. So, strong partial $n$-metric spaces are indeed partial $n$-metric spaces. We show later (cf. Lemma~\ref{sepequiv}) that the separation condition {\bf(sep)} is in fact equivalent to the more general condition
\begin{axiom}
\ax{(sep')} $G(\langle x\rangle_1^n) = \cdots = G(\langle x\rangle_1^{n-k}, \langle y\rangle_1^k) =  \cdots = G(\langle y\rangle_1^n) \iff x = y.$
\end{axiom}
In Section 3, we also show that there is a natural topology associated with a partial $n$-metric that makes $(X, G)$ a topological space.

Our main goal in this paper is to prove fixed point theorems for partial $n$-metric spaces. To do that, one usually needs to place some conditions on the function and/or on the underlying space. To that end, we use the following definitions, which are generalized versions of those given by Matthews \cite{SM1, SM2}.
\begin{defn}		\label{cauchylimit}
Let $(X, G)$ be a partial $n$-metric space. A sequence $\langle x_m\rangle_{m\in\bN}$ is called a {\em Cauchy sequence in $(X, G)$} if there is some $r\in\bR$ such that
$$\lim_{m_1, \ldots, m_n\to\infty}G(x_{m_1}, \ldots, x_{m_n}) = r.$$
An element $a\in X$ is called a {\em limit} of the sequence $\langle x_m\rangle_{m\in\bN}$ if
$$\lim_{m\to\infty}G(\langle a\rangle_1^{n-1}, x_m) = G(\langle a\rangle_1^n).$$
An element $a\in X$ is called a {\em special limit} of the sequence $\langle x_m\rangle_{m\in\bN}$ if
$$\lim_{m_1, \ldots, m_n\to\infty}G(x_{m_1}, \ldots, x_{m_n}) = \lim_{m\to\infty}G(\langle a\rangle_1^{n-1}, x_m) = G(\langle a\rangle_1^n).$$
Finally, a partial $n$-metric space $(X, G)$ is called {\em complete} if every Cauchy sequence $\langle x_m\rangle_{m\in\bN}$ in $(X, G)$ converges to a special limit $a\in X$.
\end{defn}

We show later (cf. Lemmas~\ref{equivlimitdefns} \& \ref{equivspcllimitdefns}) that these definitions of a limit and a special limit coincide with the topological definitions in two respective topologies developed in the next section. We also give equivalent definitions later of a Cauchy sequence (cf. Lemma~\ref{equivCauchy}) and of a special limit (cf. Lemma~\ref{equivspllimit}) that are more helpful in calculations. Although a limit of a sequence is not unique in general in a partial $n$-metric space, we show later that a special limit of a Cauchy sequence, if it exists, is in fact unique (cf.~Lemma~\ref{spllimitunique}). We also need some conditions on the function as given by the following definition. A remark about the notation: we write $fx$ for $f(x)$ and $f^mx$ for $f^{\circ m}(x)$.

\begin{defn}		\label{contmap}
Let $(X, G)$ be a partial $n$-metric space and $f:X\to X$ be a map. We say $f$ is {\em non-expansive} if it satisfies
$$G(\langle fx_i\rangle_{i=1}^n)\le G(\langle x_i\rangle_{i=1}^n)\;\;\;\;\;\;\forall \langle x_1, \ldots, x_n\rangle \in X^n.$$
\end{defn}
%
%
%
\begin{defn}		\label{cauchyatx0}
Given a partial $n$-metric space $(X, G)$, an element $x_0\in X$, and a map $f:X\to X$, we say $f$ is {\em Cauchy at $x_0$} if $\langle f^mx_0\rangle_{m\in\bN}$ (the orbit of $x_0$ under $f$) is a Cauchy sequence in $(X, G)$. And we say $f$ is {\em Cauchy} if $f$ is Cauchy at every $x\in X$.
\end{defn}

\begin{defn}		\label{orbcont}
Given a partial $n$-metric space $(X, G)$, elements $x_0, z_0\in X$, and a map $f:X\to X$, we say $f$ is {\em orbitally continuous at $x_0$ for $z_0$} if
$$z_0 \mbox{ is a limit of } \langle f^mx_0\rangle_{m\in\bN} \implies fz_0 \mbox{ is a limit of } \langle f^{m}x_0\rangle_{m\in\bN}$$
i.e.,
$$\lim_{m\to\infty}G(\langle z_0\rangle_1^{n-1}, f^mx_0) = G(\langle z_0\rangle_1^n) \implies \lim_{m\to\infty}G(\langle fz_0\rangle_1^{n-1}, f^{m}x_0) = G(\langle fz_0\rangle_1^n).$$

We say $f$ is {\em orbitally continuous at $x_0$} if $f$ is orbitally continuous at $x_0$ for every $z\in X$.
And, we say $f$ is {\em orbitally continuous} if it is orbitally continuous at every $x\in X$.
\end{defn}

The following are the main results of this paper.
\begin{thm}[Cauchy Mapping Theorem for Partial $n$-Metric Spaces]		\label{cauchyorbcontnexpfpt}
Let $(X, G)$ be a partial $n$-metric space, $x_0\in X$ be an element, and $f:X\to X$ be a map such that $f$ is Cauchy at $x_0$ with special limit $a\in X$. Further assume one of the following holds:
\begin{enumerate}
\item $f$ is non-expansive and orbitally continuous at $x_0$ for $a$;
\item $f$ is orbitally continuous at $x_0$ for $a$ and the self distances of $(X, G)$ are bounded below by $G(\langle fa\rangle_1^n)$;
\item $f$ is non-expansive and the self distances of $(X, G)$ are bounded below by $G(\langle a\rangle_1^n)$.
\end{enumerate}
Then $a$ is a fixed point of $f$.
\end{thm}

As it turns out, things are much simpler for a strong partial $n$-metric space as we need fewer conditions for a fixed point to exist.
\begin{thm}[Cauchy Mapping Theorem for Strong Partial $n$-Metric Spaces]		\label{strongcauchyorbcontnexpfpt}
Let $(X, G)$ be a strong partial $n$-metric space, $x_0\in X$, and $f:X\to X$ be a map such that $f$ is Cauchy at $x_0$ with special limit $a\in X$. Further assume one of the following holds:
\begin{enumerate}
\item $f$ is non-expansive;
\item $f$ is orbitally continuous at $x_0$ for $a$.
\end{enumerate}
Then $a$ is a fixed point of $f$.
\end{thm}

\section{Topology}
Let $(X, G)$ be a partial $n$-metric space. We define an {\bf open ball} as:
$$B_\epsilon(x) := \{y\mid G(\langle x\rangle_1^{n-1}, y) - G(\langle x\rangle_1^n) < \epsilon\}$$
for $x\in X$ and $\epsilon\in\bR^{> 0}$. It is easy to see that these balls are nonempty for $\epsilon > 0$, and empty for $\epsilon \le 0$. We now show that these balls form a basis for a $T_0$ topology on $X$, called the {\em partial n-metric topology}, and is denoted by $\tau[G]$. Since the set of positive rational numbers $\bQ^{> 0}$ is dense in $\bR^{> 0}$, it then follows that every point $x\in X$ has a countable local base given by $\{B_q(x)\mid q\in\bQ^{> 0}\}$. Hence, $(X, G)$ is first countable as well.

\begin{lem}		\label{basislem}
The set $\{B_\epsilon(x)\mid x\in X, \epsilon\in\bR^{>0}\}$ forms a basis for a topology on $X$.
\end{lem}
\begin{proof}
It suffices to show that for every ball $B_\epsilon(x)$ and every $y\in B_\epsilon(x)$, there is $\delta > 0$ such that $y \in B_\delta(y)\subseteq B_\epsilon(x)$. So let $y\in B_\epsilon(x)$. Then $G(\langle x\rangle_1^{n-1}, y) - G(\langle x\rangle_1^n) < \epsilon$. Set
$$\delta := \epsilon - G(\langle x\rangle_1^{n-1}, y) + G(\langle x\rangle_1^n).$$
Then $\delta > 0$, and hence $y\in B_\delta(y)$. Now let $z\in B_\delta(y)$, i.e., $G(\langle y\rangle_1^{n-1}, z) - G(\langle y\rangle_1^n) < \delta$. Then
\begin{eqnarray*}
G(\langle x\rangle_1^{n-1}, z) - G(\langle x\rangle_1^n) & \le & G(\langle x\rangle_1^{n-1}, y) + G(\langle y\rangle_1^{n-1}, z) - G(\langle y\rangle_1^n) - G(\langle x\rangle_1^n) \\
& < & G(\langle x\rangle_1^{n-1}, y) - G(\langle x\rangle_1^n) + \delta \\
& = & \epsilon.
\end{eqnarray*}
Thus, $z\in B_\epsilon(x)$, and hence $B_\delta(y)\subseteq B_\epsilon(x)$.
\end{proof}

\begin{lem}		\label{pnmetricT0top}
The partial $n$-metric topology $\tau[G]$ on a partial $n$-metric space $(X, G)$ is $T_0$. If $(X, G)$ is a strong partial $n$-metric space, then $\tau[G]$ is $T_1$.
\end{lem}
\begin{proof}
Let $x\not=y\in X$. Set $\epsilon_x := G(\langle x\rangle_1^{n-1}, y) - G(\langle x\rangle_1^n)$ and $\epsilon_y := G(\langle y\rangle_1^{n-1}, x) - G(\langle y\rangle_1^n)$. 

If $(X, G)$ is a partial $n$-metric space, then either $\epsilon_x > 0$ or $\epsilon_y > 0$ (by condition {\bf(sep)}). Consequently, either $y\not\in B_{\epsilon_x}(x)$ or $x\not\in B_{\epsilon_y}(y)$. Thus, $\tau[G]$ is $T_0$.

If $(X, G)$ is a strong partial $n$-metric space, then both $\epsilon_x  > 0$ and $\epsilon_y > 0$ (by condition {\bf(sssd)}). Consequently, $y\not\in B_{\epsilon_x}(x)$ and $x\not\in B_{\epsilon_y}(y)$. Thus, $\tau[G]$ is $T_1$.
\end{proof}

As promised in Section 2, we now show that the definition of a limit of a sequence (cf. Definition~\ref{cauchylimit}) agrees with the topological definition of a limit in the topology $\tau[G]$.
\begin{lem}		\label{equivlimitdefns}
Let $(X, G)$ be a partial $n$-metric space, $\langle x_m\rangle_{m\in\bN}$ be a sequence in $X$, and $a\in X$. Then
$$a \mbox{ is a topological limit of } \langle x_m\rangle_{m\in\bN} \mbox{ in } \tau[G] \iff \lim_{m\to\infty} G(\langle a\rangle_1^{n-1}, x_m) = G(\langle a\rangle_1^n).$$
\end{lem}
\begin{proof}
\begin{eqnarray*}
& & a \mbox{ is a topological limit of } \langle x_m\rangle_{m\in\bN} \mbox{ in } \tau[G]\\
& \iff & \forall\epsilon > 0 \;\exists m_0\in\bN \mbox{ such that } x_m\in B_\epsilon(a) \mbox{ for all } m\ge m_0 \\
& \iff & \forall\epsilon > 0 \;\exists m_0\in\bN \mbox{ such that } 0 \le G(\langle a\rangle_1^{n-1}, x_m) - G(\langle a\rangle_1^n) < \epsilon \mbox{ for all } m\ge m_0 \\
& \iff & \lim_{m\to\infty} G(\langle a\rangle_1^{n-1}, x_m) = G(\langle a\rangle_1^n).
\end{eqnarray*}
\end{proof}

We now show that a partial metric space $(X, p)$ gives rise to a partial $n$-metric.
\begin{example}
Let $(X, p)$ be a partial metric space. Define $G:X^n\to\bR$ as follows:
$$G(\langle x_i\rangle_1^n) = \sum_{i=1}^{n-1}\sum_{j = i+1}^{n} p(x_i, x_j).$$
In particular,
$$G(\langle x\rangle_1^n) = \frac{n(n-1)}{2} p(x, x) \;\;\;\;\mbox{and}\;\;\;\; G(\langle x\rangle_1^{n-1}, y) = (n-1)p(x, y) + \frac{(n-1)(n-2)}{2} p(x, x).$$
We leave it to the reader to verify that $(X, G)$ is indeed a partial $n$-metric space. And if $(X, p)$ is a strong partial metric space to start with, then $(X, G)$ turns out to be a strong partial $n$-metric space.
\end{example}

\subsection{Associated metric}
In a partial $n$-metric space $(X, G)$, one can define a metric $d_G:X\times X\to\bR^{\ge 0}$ on $X$ as follows:
$$d_G(x, y) := G(\langle x\rangle_1^{n-1}, y) - G(\langle x\rangle_1^n) + G(\langle y\rangle_1^{n-1}, x) - G(\langle y\rangle_1^n).$$
\begin{lem}		\label{partialnmetricdistance}
$(X, d_G)$ (with $d_G$ as defined above) is a metric space.
\end{lem}
\begin{proof}
Let $x, y\in X$. By condition {\bf(ssd)}, we have 
$$G(\langle x\rangle_1^{n-1}, y) - G(\langle x\rangle_1^n) \ge 0\;\;\;\;\;\wedge\;\;\;\;\; G(\langle y\rangle_1^{n-1}, x) - G(\langle y\rangle_1^n) \ge 0.$$
Consequently, $d_G(x, y) \ge 0$ for all $x, y\in X$. Moreover,
\begin{eqnarray*}
d_G(x, y) = 0 & \iff & G(\langle x\rangle_1^{n-1}, y) - G(\langle x\rangle_1^n) + G(\langle y\rangle_1^{n-1}, x) - G(\langle y\rangle_1^n) = 0 \\
& \iff & G(\langle x\rangle_1^{n-1}, y) - G(\langle x\rangle_1^n) = 0 \;\;\;\wedge\;\;\; G(\langle y\rangle_1^{n-1}, x) - G(\langle y\rangle_1^n) = 0 \\
& \iff & x = y \;\;\;\;\;\mbox{ (by condition {\bf(sep)}).}
\end{eqnarray*}
The condition of symmetry, i.e., $d_G(x, y) = d_G(y, x)$ for all $x, y\in X$, is obvious.\newline
Finally, by condition {\bf(ptri)} of $(X, G)$, we have for any $z\in X$
\begin{eqnarray*}
d_G(x, y) & = & G(\langle x\rangle_1^{n-1}, y) - G(\langle x\rangle_1^n) + G(\langle y\rangle_1^{n-1}, x) - G(\langle y\rangle_1^n) \\
& \le & G(\langle x\rangle_1^{n-1}, z) + G(\langle z\rangle_1^{n-1}, y) - G(\langle z\rangle_1^n) - G(\langle x\rangle_1^n) \\
& & +\; G(\langle y\rangle_1^{n-1}, z) + G(\langle z\rangle_1^{n-1}, x) - G(\langle z\rangle_1^n) - G(\langle y\rangle_1^n) \\
& = & d_G(x, z) + d_G(z, y).
\end{eqnarray*}
Thus, $d_G$ satisfies the triangle inequality. And hence, $(X, d_G)$ is a metric space.
\end{proof}

As promised in Section 2 again, we now show that the definition of a special limit of a Cauchy sequence (cf. Definition~\ref{cauchylimit}) agrees with the topological definition of a limit in the metric topology $\tau[d]$.
\begin{lem}		\label{equivspcllimitdefns}
Let $(X, G)$ be a partial $n$-metric space, $\langle x_m\rangle_{m\in\bN}$ be a Cauchy sequence in $(X, G)$, and $a\in X$. Then
\begin{eqnarray*}
& & a \mbox{ is a topological limit of } \langle x_m\rangle_{m\in\bN} \mbox{ in } \tau[d] \\
& \iff & \lim_{m_1, \ldots, m_n\to\infty} G(\langle x_{m_i}\rangle_{i=1}^n) = \lim_{m\to\infty} G(\langle a\rangle_1^{n-1}, x_m) = G(\langle a\rangle_1^n).
\end{eqnarray*}
In particular, $(X, G)$ is a complete partial $n$-metric space if and only if it is complete in the usual sense with respect to the metric topology $\tau[d]$.
\end{lem}
\begin{proof}
\begin{eqnarray*}
& & a \mbox{ is a topological limit of } \langle x_m\rangle_{m\in\bN} \mbox{ in } \tau[d]\\
& \iff & \forall\epsilon > 0 \;\exists m_0\in\bN \mbox{ such that for all } m\ge m_0 \\
& & 0 \le G(\langle a\rangle_1^{n-1}, x_m) - G(\langle a\rangle_1^n) + G(\langle x_m\rangle_1^{n-1}, a) - G(\langle x_m\rangle_1^n) < \epsilon \\
& \iff & \lim_{m\to\infty} G(\langle a\rangle_1^{n-1}, x_m) = G(\langle a\rangle_1^n) \mbox{ and } \lim_{m\to\infty} (G(\langle x_m\rangle_1^{n-1}, a) - G(\langle x_m\rangle_1^n)) = 0 \\
& \iff & \lim_{m_1, \ldots, m_n\to\infty} G(\langle x_{m_i}\rangle_{i=1}^n) = G(\langle a\rangle_1^n) = \lim_{m\to\infty} G(\langle a\rangle_1^{n-1}, x_m).
\end{eqnarray*}
The last equivalence follows by Lemma~\ref{equivspllimit}.
\end{proof}

\section{Zoo of Limit Lemmas}
In this section, we prove a whole zoo of limit lemmas, and consequently show that the partial $n$-metric $G$ is continuous in all its variables. We start by proving the following basic inequalities which will be useful to us for proving these limit theorems.

\begin{lem}		\label{basicinequality}
Let $(X, G)$ be a partial $n$-metric space. Then for any $x_1, \ldots, x_n, y_1, \ldots, y_n,$ $x, y\in X$ and $1\le k\le n$, we have the following:
\begin{itemize}
\item[(a)] $G(\langle x_i\rangle_{i=1}^k, \langle z_i\rangle_{i=1}^{n-k}) \le G(\langle y_i\rangle_{i=1}^k, \langle z_i\rangle_{i=1}^{n-k}) + \sum_{j = 1}^{k} (G(\langle y_j\rangle_1^{n-1}, x_j) - G(\langle y_j\rangle_1^n))$.
\item[(b)] $G(\langle x_i\rangle_{i=1}^n) \le G(\langle y_i\rangle_{i=1}^n) + \sum_{j = 1}^{n} (G(\langle y_j\rangle_1^{n-1}, x_j) - G(\langle y_j\rangle_1^n))$.
\item[(c)] $G(\langle x\rangle_1^{n-1}, y) \le (n-1)\, G(\langle y\rangle_1^{n-1}, x) - (n-2)\, G(\langle y\rangle_1^n)$.
\item[(d)] $G(\langle x_i\rangle_{i=1}^n) \le \sum_{j=1}^n G(\langle y\rangle_1^{n-1}, x_j) - (n-1)G(\langle y\rangle_1^n)$.
\end{itemize}
\end{lem}
\begin{proof}
(a) By repeated application of condition {\bf(ptri)}, we have that
\begin{eqnarray*}
G(\langle x_i\rangle_{i=1}^k, \langle z_i\rangle_{i=1}^{n-k}) & \le & G(\langle x_i\rangle_{i=1}^{k-1}, \langle z_i\rangle_{i=1}^{n-k}, y_k) + G(\langle y_k\rangle_1^{n-1}, x_k) - G(\langle y_k\rangle_1^n) \\
& \le & G(\langle x_i\rangle_{i=1}^{k-2}, \langle z_i\rangle_{i=1}^{n-k}, y_{k-1}, y_k) + \sum_{j = k-1}^k \Big(G(\langle y_j\rangle_1^{n-1}, x_j) - G(\langle y_j\rangle_1^n)\Big) \\
& \le & \cdots \\
& \le &  G(x_1, \langle z_i\rangle_{i=1}^{n-k}, \langle y_i\rangle_{i=2}^k) + \sum_{j = 2}^k \Big(G(\langle y_j\rangle_1^{n-1}, x_j) - G(\langle y_j\rangle_1^n)\Big)  \\
& \le & G(\langle y_i\rangle_{i=1}^k, \langle z_i\rangle_{i=1}^{n-k}) + \sum_{j = 1}^k \Big(G(\langle y_j\rangle_1^{n-1}, x_j) - G(\langle y_j\rangle_1^n)\Big).
\end{eqnarray*}
(b) This follows from (a) by taking $k = n$. \newline
(c) This follows from (b) by setting 
$$x_1 = \cdots = x_{n-1} = x,\;\;\;\;\; x_n = y,\;\;\;\;\; y_1 = x,\;\;\;\;\; y_2 = \cdots = y_n = y.$$
(d) This follows from (b) by setting $y_1 = \cdots = y_n = y$.
\end{proof}

\subsection{Properties of Limits.}
Now we list a set of properties of limits.
\begin{lem}		\label{limitprop}
Let $(X, G)$ be a partial $n$-metric space. Let $\langle x_m\rangle_{m\in\bN}$ be a sequence in $(X, G)$ with a limit $a\in X$. Let $b_1, \ldots, b_{n-1}\in X$. Then, provided the following limits exist, we have
\begin{itemize}
\item[(a)] $\lim_{m_1, \ldots, m_k\to\infty} G(\langle x_{m_i}\rangle_{i=1}^k, \langle b_i\rangle_{i=1}^{n-k}) \le G(\langle a\rangle_1^k, \langle b_i\rangle_{i=1}^{n-k})\;\;\;\;\;$ for all $\;1\le k\le n$.
\item[(b)] $\lim_{m_1, \ldots, m_n\to\infty} G(\langle x_{m_i}\rangle_{i=1}^n) \le G(\langle a\rangle_1^n)$.
\item[(c)] $\lim_{m_1, \ldots, m_k\to\infty} G(\langle x_{m_i}\rangle_{i=1}^k, \langle a\rangle_1^{n-k}) \le G(\langle a\rangle_1^n)\;\;\;\;\;$ for all $\;1\le k\le n$.
\item[(d)] $\lim_{m\to\infty} G(\langle x_m\rangle_1^{n-1}, a) \le G(\langle a\rangle_1^n)$.
\end{itemize}
\end{lem}
\begin{proof}
(a) By Lemma~\ref{basicinequality}(a), we have for all $m_1, \ldots, m_k\in\bN$
$$G(\langle x_{m_i}\rangle_{i=1}^k, \langle b_i\rangle_{i=1}^{n-k}) \le G(\langle a\rangle_1^k, \langle b_i\rangle_{i=1}^{n-k}) + \sum_{j = 1}^{k} (G(\langle a\rangle_1^{n-1}, x_{m_j}) - G(\langle a\rangle_1^n)).$$
Taking the limit as $m_1, \ldots, m_k\to\infty$ and using Definition~\ref{cauchylimit}, we get that
\begin{eqnarray*}
& & \lim_{m_1, \ldots, m_k\to\infty} G(\langle x_{m_i}\rangle_{i=1}^k, \langle b_i\rangle_{i=1}^{n-k}) \\
& \le & G(\langle a\rangle_1^k, \langle b_i\rangle_{i=1}^{n-k}) + \sum_{j = 1}^{k} (\lim_{m_j\to\infty}G(\langle a\rangle_1^{n-1}, x_{m_j}) - G(\langle a\rangle_1^n)) \\
& = & G(\langle a\rangle_1^k, \langle b_i\rangle_{i=1}^{n-k}).
\end{eqnarray*}

(b) This inequality follows from (a) by taking $k = n$.

(c) This inequality follows from (a) by setting $b_1 = \cdots = b_{n - k} = a$, for any $1\le k < n$.

(d) This inequality follows from (c) by setting $k = n-1$ and $m_1 = \cdots = m_{n-1} = m$.
\end{proof}

\subsection{Properties of Special Limits.}
In this subsection, we show that if we restrict ourselves to Cauchy sequences and special limits of such sequences, then all the limits mentioned in the previous subsection exist and all the inequalities become equalities.
\begin{lem}		\label{speciallimitprop}
Let $(X, G)$ be a partial $n$-metric space and $\langle x_m\rangle_{m\in\bN}$ be a Cauchy sequence in $(X, G)$ with a special limit $a\in X$. Let $b_1, \ldots, b_{n-1}\in X$. Then 
\begin{itemize}
\item[(a)] $\lim_{m\to\infty} G(\langle x_m\rangle_1^{n-1}, a) = G(\langle a\rangle_1^n)$.
\item[(b)] $\lim_{m_1, \ldots, m_k\to\infty} G(\langle x_{m_i}\rangle_{i=1}^k, \langle b_i\rangle_{i=1}^{n-k}) = G(\langle a\rangle_1^k, \langle b_i\rangle_{i=1}^{n-k})\;\;\;\;\;$ for all $\;1\le k\le n$.
\item[(c)] $\lim_{m_1, \ldots, m_n\to\infty} G(\langle x_{m_i}\rangle_{i=1}^n) = G(\langle a\rangle_1^n)$.
\item[(d)] $\lim_{m_1, \ldots, m_k\to\infty} G(\langle x_{m_i}\rangle_{i=1}^k, \langle a\rangle_1^{n-k}) = G(\langle a\rangle_1^n)\;\;\;\;\;$ for all $\;1\le k\le n$.
\end{itemize}
\end{lem}
\begin{proof}
Since $\langle x_m\rangle_{m\in\bN}$ is a Cauchy sequence with special limit $a$, we have
$$\lim_{m_1, \ldots, m_n\to\infty} G(\langle x_{m_i}\rangle_{i=1}^n) = \lim_{m\to\infty} G(\langle a\rangle_1^{n-1}, x_m) = G(\langle a\rangle_1^n).$$

(a) Fix $\epsilon > 0$. By Lemma~\ref{basicinequality}(c), we have for any $m\in\bN$
$$G(\langle x_m\rangle_1^{n-1}, a) \le (n - 1)\,G(\langle a\rangle_1^{n-1}, x_m) - (n - 2)\,G(\langle a\rangle_1^n).$$
Set $\epsilon' := \frac{\epsilon}{n-1}$, and choose $M_1$ large enough such that $G(\langle a\rangle_1^{n-1}, x_m)\le G(\langle a\rangle_1^n) + \epsilon'$ for all $m\ge M_1$. Thus, for all $m\ge M_1$, we have
$$G(\langle x_m\rangle_1^{n-1}, a) \le (n-1)\, (G(\langle a\rangle_1^n) + \epsilon') - (n - 2)\,G(\langle a\rangle_1^n) = G(\langle a\rangle_1^n) + (n-1)\,\epsilon' = G(\langle a\rangle_1^n) + \epsilon.$$

Conversely, by Lemma~\ref{basicinequality}(c) again, we have for any $m\in\bN$
$$G(\langle a\rangle_1^{n-1}, x_m) \le (n - 1)\,G(\langle x_m\rangle_1^{n-1}, a) - (n - 2)\,G(\langle x_m\rangle_1^n),$$
that is,
$$G(\langle a\rangle_1^{n-1}, x_m) +  (n - 2)\,G(\langle x_m\rangle_1^n) \le (n - 1)\,G(\langle x_m\rangle_1^{n-1}, a).$$
Set $\epsilon'' := \epsilon$, and choose $M_2$ large enough such that $G(\langle a\rangle_1^{n-1}, x_m)\ge G(\langle a\rangle_1^n) - \epsilon''$ and $G(\langle x_m\rangle_1^n)\ge G(\langle a\rangle_1^n) - \epsilon''$ for all $m\ge M_2$. Thus, for all $m\ge M_2$, we have
\begin{eqnarray*}
(n-1)\,G(\langle x_m\rangle_1^{n-1}, a) & \ge & G(\langle a\rangle_1^{n-1}, x_m) +  (n - 2)\,G(\langle x_m\rangle_1^n) \\
& \ge & (G(\langle a\rangle_1^n) - \epsilon'') + (n - 2)\,(G(\langle a\rangle_1^n) - \epsilon'') \\
& = & (n-1)\,(G(\langle a\rangle_1^n) - \epsilon'') \\
\implies G(\langle x_m\rangle_1^{n-1}, a) & \ge & G(\langle a\rangle_1^n) - \epsilon'' = G(\langle a\rangle_1^n) - \epsilon.
\end{eqnarray*}
Setting $M := \max\{M_1, M_2\}$, we obtain that for all $m\ge M$
$$0\le |G(\langle x_m\rangle_1^{n-1}, a) - G(\langle a\rangle_1^n)| \le \epsilon.$$
Since $0 < \epsilon$ is arbitrary, it follows that $\lim_{m\to\infty} G(\langle x_m\rangle_1^{n-1}, a) = G(\langle a\rangle_1^n)$.\newline

(b) Fix $\epsilon > 0$. By Lemma~\ref{basicinequality}(a), we have for all $m_1, \ldots, m_n\in\bN$
$$G(\langle a\rangle_1^k, \langle b_i\rangle_{i=1}^{n-k}) \le G(\langle x_{m_i}\rangle_{i=1}^k, \langle b_i\rangle_{i=1}^{n-k}) + \sum_{j = 1}^{k} (G(\langle x_{m_j}\rangle_1^{n-1}, a) - G(\langle x_{m_j}\rangle_1^n)).$$
Set $\epsilon' := \frac{\epsilon}{2n}$. Using part (a), choose $M_1$ large enough such that $G(\langle x_m\rangle_1^{n-1}, a)\le G(\langle a\rangle_1^n) + \epsilon'$ and $G(\langle x_m\rangle_1^n)\ge G(\langle a\rangle_1^n) - \epsilon'$ for all $m\ge M_1$. Thus, for all $m_1, \ldots, m_n\ge M_1$, we have
\begin{eqnarray*}
G(\langle a\rangle_1^k, \langle b_i\rangle_{i=1}^{n-k}) & \le & G(\langle x_{m_i}\rangle_{i=1}^k, \langle b_i\rangle_{i=1}^{n-k}) + \sum_{j = 1}^{k} \Big[(G(\langle a\rangle_1^n) + \epsilon') - (G(\langle a\rangle_1^n) - \epsilon')\Big] \\
& = & G(\langle x_{m_i}\rangle_{i=1}^k, \langle b_i\rangle_{i=1}^{n-k}) + 2k\epsilon' \\
& \le & G(\langle x_{m_i}\rangle_{i=1}^k, \langle b_i\rangle_{i=1}^{n-k}) + 2n\epsilon' \\
& = & G(\langle x_{m_i}\rangle_{i=1}^k, \langle b_i\rangle_{i=1}^{n-k}) + \epsilon \\
\implies G(\langle a\rangle_1^k, \langle b_i\rangle_{i=1}^{n-k}) - \epsilon & \le &  G(\langle x_{m_i}\rangle_{i=1}^k, \langle b_i\rangle_{i=1}^{n-k}).
\end{eqnarray*}

Conversely, by Lemma~\ref{basicinequality}(a) again, we have for all $m_1, \ldots, m_k\in\bN$
$$G(\langle x_{m_i}\rangle_{i=1}^k, \langle b_i\rangle_{i=1}^{n-k}) \le G(\langle a\rangle_1^k, \langle b_i\rangle_{i=1}^{n-k}) + \sum_{j = 1}^{k} (G(\langle a\rangle_1^{n-1}, x_{m_j}) - G(\langle a\rangle_1^n)).$$
Set $\epsilon'' := \frac{\epsilon}{n}$, and choose $M_2$ large enough such that $G(\langle a\rangle_1^{n-1}, x_m)\le G(\langle a\rangle_1^n) + \epsilon''$ for all $m\ge M_2$. Then for all $m_1, \ldots, m_k\ge M_2$, we have
\begin{eqnarray*}
G(\langle x_{m_i}\rangle_{i=1}^k, \langle b_i\rangle_{i=1}^{n-k}) & \le & G(\langle a\rangle_1^k, \langle b_i\rangle_{i=1}^{n-k}) + \sum_{j = 1}^{k} \Big[(G(\langle a\rangle_1^n) + \epsilon'') - G(\langle a\rangle_1^n)\Big] \\
& = & G(\langle a\rangle_1^k, \langle b_i\rangle_{i=1}^{n-k}) + k\epsilon'' \\
& \le & G(\langle a\rangle_1^k, \langle b_i\rangle_{i=1}^{n-k}) + n\epsilon'' \\
& = & G(\langle a\rangle_1^k, \langle b_i\rangle_{i=1}^{n-k}) + \epsilon.
\end{eqnarray*}
Setting $M := \max\{M_1, M_2\}$, we obtain that for all $m_1, \ldots, m_k\ge M$
$$0\le |G(\langle x_{m_i}\rangle_{i=1}^k, \langle b_i\rangle_{i=1}^{n-k}) - G(\langle a\rangle_1^k, \langle b_i\rangle_{i=1}^{n-k})| \le \epsilon.$$
Since $0 < \epsilon$ is arbitrary, it follows that $\lim_{m\to\infty} G(\langle x_{m_i}\rangle_{i=1}^k, \langle b_i\rangle_{i=1}^{n-k})  = G(\langle a\rangle_1^k, \langle b_i\rangle_{i=1}^{n-k})$.\newline

(c) This equality follows from (b) by taking $k = n$.

(d) This equality follows from (b) by setting $b_1 = \cdots = b_{n - k} = a$, for any $1\le k < n$.
\end{proof}

Finally, we get the uniqueness of special limits as a corollary.
\begin{lem}		\label{spllimitunique}
For each partial $n$-metric space $(X, G)$ and each Cauchy sequence $\langle x_m\rangle_{m\in\bN}$ in $(X, G)$, there is at most one special limit of $\langle x_m\rangle_{m\in\bN}$ in $X$.
\end{lem}
\begin{proof}
Let $a$ and $b$ be two special limits of $\langle x_m\rangle_{m\in\bN}$ in $X$. By Lemma~\ref{speciallimitprop}(a), we have
\begin{eqnarray*}
& & G(\langle a\rangle_1^n) = \lim_{m\to\infty} G(\langle a\rangle_1^{n-1}, x_m) = G(\langle a\rangle_1^{n-1}, b) \\
& & G(\langle b\rangle_1^n) = \lim_{m\to\infty} G(\langle b\rangle_1^{n-1}, x_m) = G(\langle b\rangle_1^{n-1}, a).
\end{eqnarray*}
By condition {\bf(sep)}, it then follows that $a = b$.
\end{proof}

To end this section, we give equivalent definitions of a Cauchy sequence and a special limit of a Cauchy sequence and the equivalence of {\bf(sep)} and {\bf(sep')} as promised in Section 2. An equivalent condition for being a Cauchy sequence is the following.
\begin{lem}		\label{equivCauchy}
Let $\langle x_m\rangle_{m\in\bN}$ be a sequence in a partial $n$-metric space $(X, G)$ and $r\in\bR$. Then
$$\lim_{m_1, \ldots, m_n\to\infty} G(\langle x_{m_i}\rangle_{i = 1}^n) = r \iff \lim_{m_1, m_2\to\infty} G(\langle x_{m_1}\rangle_1^{n-1}, x_{m_2}) = r.$$
\end{lem}
\begin{proof}
The left to right direction is trivial. So we prove the right to left direction. Assume $\lim_{m_1, m_2\to\infty} G(\langle x_{m_1}\rangle_1^{n-1}, x_{m_2}) = r$. Fix $\epsilon > 0$. Set $\epsilon' := \frac{\epsilon}{2n-1}$ and choose $M_1\in\bN$ large enough such that for all $m_1, m_2\ge M_1$, we have
$$r - \epsilon' < G(\langle x_{m_1}\rangle_1^{n-1}, x_{m_2}) < r + \epsilon'.$$  
By Lemma~\ref{basicinequality}(d), we then have for all $m_1, \ldots, m_n\ge M_1$
\begin{eqnarray*}
G(\langle x_{m_i}\rangle_{i=1}^n) & \le & \sum_{t=1}^n G(\langle x_{m_1}\rangle_1^{n-1}, x_{m_t}) - (n - 1)\,G(\langle x_{m_1}\rangle_1^n) \\
& < & \sum_{t=1}^n (r + \epsilon') - (n - 1)\,(r - \epsilon') \\
& = & r + (2n - 1)\epsilon' \\
& = & r + \epsilon.
\end{eqnarray*}
Conversely, set $\epsilon'' := \frac{\epsilon}{2n+1}$, and choose $M_2\in\bN$ large enough such that for all $m_1, m_2\ge M_2$ 
$$r - \epsilon'' < G(\langle x_{m_1}\rangle_1^{n-1}, x_{m_2}) < r + \epsilon''.$$
By Lemma~\ref{basicinequality}(a), we then have for all $m_1, \ldots, m_n\ge M_2$
\begin{eqnarray*}
G(\langle x_{m_1}\rangle_1^n) & \le & G(\langle x_{m_i}\rangle_{i=1}^n) + \sum_{t=1}^n \Big[G(\langle x_{m_t}\rangle_1^{n-1}, x_{m_1}) - G(\langle x_{m_t}\rangle_1^n)\Big] \\
\implies G(\langle x_{m_i}\rangle_{i=1}^n) & \ge & G(\langle x_{m_1}\rangle_1^n) + \sum_{t=1}^n G(\langle x_{m_t}\rangle_1^n) - \sum_{t=1}^n G(\langle x_{m_t}\rangle_1^{n-1}, x_{m_1}) \\
& > & (r - \epsilon'') + \sum_{t = 1}^n (r - \epsilon'') - \sum_{t = 1}^n (r + \epsilon'') \\
& = & r - (2n + 1)\epsilon'' \\
& = & r - \epsilon.
\end{eqnarray*}
Setting $M := \max\{M_1, M_2\}$, we obtain that for all $m_1, \ldots, m_n\ge M$
$$r - \epsilon < G(\langle x_{m_i}\rangle_{i=1}^n) < r + \epsilon.$$
Since $0 < \epsilon$ is arbitrary, we obtain $\lim_{m_1, \ldots, m_n\to\infty} G(\langle x_{m_i}\rangle_{i=1}^n) = r$.
\end{proof}

An equivalent condition for being a special limit of a Cauchy sequence is the following.
\begin{lem}		\label{equivspllimit}
An element $a\in X$ is a special limit of the Cauchy sequence $\langle x_m\rangle_{m\in\bN}$ if and only if $\lim_{m\to\infty} G(\langle a\rangle_1^{n-1}, x_m) = G(\langle a\rangle_1^n) \mbox{ and } \lim_{m\to\infty} (G(\langle x_m\rangle_1^{n-1}, a) - G(\langle x_m\rangle_1^n)) = 0$.
\end{lem}
\begin{proof}
If $a$ is a special limit of the Cauchy sequence $\langle x_m\rangle_{m\in\bN}$, then the first condition is trivially satisfied and the second follows by Lemma~\ref{speciallimitprop}(a).

For the converse, assume the two given conditions hold. We want to show 
$$\lim_{m_1, \ldots, m_n\to\infty}G(\langle x_{m_i}\rangle_{i=1}^n) = G(\langle a\rangle_1^n).$$
Since the first condition is satisfied, we have that $a$ is a limit of the sequence $\langle x_m\rangle_{m\in\bN}$. Since $\langle x_m\rangle_{m\in\bN}$ is Cauchy, we know that $\lim_{m_1, \ldots, m_n\to\infty} G(\langle x_{m_i}\rangle_{i=1}^n)$ exists, and thus by Lemma~\ref{limitprop}(b), we have 
$$\lim_{m_1, \ldots, m_n\to\infty}G(\langle x_{m_i}\rangle_{i=1}^n) \le G(\langle a\rangle_1^n).$$
On the other hand, by Lemma~\ref{basicinequality}(b), we also have
\begin{eqnarray*}
G(\langle a\rangle_1^n) & \le & G(\langle x_{m_i}\rangle_{i=1}^n) + \sum_{j = 1}^n (G(\langle x_{m_j}\rangle_1^{n-1}, a) - G(\langle x_{m_j}\rangle_1^n)) \\
\implies G(\langle a\rangle_1^n) & \le & \lim_{m_1, \ldots, m_n\to\infty} G(\langle x_{m_i}\rangle_{i=1}^n) + \sum_{j = 1}^n \lim_{m_j\to\infty}(G(\langle x_{m_j}\rangle_1^{n-1}, a) - G(\langle x_{m_j}\rangle_1^n)) \\
& = & \lim_{m_1, \ldots, m_n\to\infty} G(\langle x_{m_i}\rangle_{i=1}^n) \;\;\;\;\;\;\;\;\;\;\;\;\;\;\;\;\;\;\;\;\mbox{(by the second condition).}
\end{eqnarray*}
Thus, $\lim_{m_1, \ldots, m_n\to\infty}G(\langle x_{m_i}\rangle_{i=1}^n) = G(\langle a\rangle_1^n)$, and hence $a$ is a special limit of $\langle x_m\rangle_{m\in\bN}$.
\end{proof}

Finally, we have the equivalence of {\bf(sep)} and {\bf(sep')}.
\begin{lem}		\label{sepequiv}
The condition {\bf(sep)} is equivalent to the following condition:
\begin{axiom}
\ax{(sep')} $G(\langle x\rangle_1^n) = \cdots = G(\langle x\rangle_1^{n-k}, \langle y\rangle_1^k) =  \cdots = G(\langle y\rangle_1^n) \iff x = y.$
\end{axiom}
\end{lem}
\begin{proof}
It is trivial to check that {\bf(sep)} $\implies$ {\bf(sep')}.

For the converse, assume that $G(\langle x \rangle_1^{n-1}, y) = G(\langle x\rangle_1^n)$ and $G(\langle y \rangle_1^{n-1}, x) = G(\langle y\rangle_1^n)$. We want to show $x = y$. By condition {\bf(sep')}, it suffices to show that 
$$G(\langle x\rangle_1^{n-k}, \langle y\rangle_1^{k}) = G(\langle y\rangle_1^n)\;\;\;\;\;\mbox{ for all } 0\le k\le n.$$
Fix $0\le k\le n$. By two applications of Lemma~\ref{basicinequality}(b), we obtain
\begin{eqnarray*}
& & G(\langle x\rangle_1^{n-k}, \langle y\rangle_1^k)\le G(\langle y\rangle_1^n) + \sum_{j = 1}^{n-k} (G(\langle y\rangle_1^{n-1}, x) - G(\langle y\rangle_1^n)) = G(\langle y\rangle_1^n) \\
& & G(\langle y\rangle_1^n)\le G(\langle x\rangle_1^{n-k}, \langle y\rangle_1^k) + \sum_{j = 1}^{n-k} (G(\langle x\rangle_1^{n-1}, y) - G(\langle x\rangle_1^n)) = G(\langle x\rangle_1^{n-k}, \langle y\rangle_1^k).
\end{eqnarray*}
Thus, $G(\langle x\rangle_1^{n-k}, \langle y\rangle_1^{k}) = G(\langle y\rangle_1^n)$ for all $0\le k\le n$. Consequently, {\bf(sep')}$\implies$ {\bf(sep)}.
\end{proof}

\section{Cauchy Mapping Theorems}
In this section, we prove our two main theorems --- Theorem~\ref{cauchyorbcontnexpfpt} and Theorem~\ref{strongcauchyorbcontnexpfpt}. But before that we need a couple of lemmas.

\begin{lem}		\label{nonexpcriterion}
Let $(X, G)$ be a partial $n$-metric space, $x_0\in X$ be an element, and $f:X\to X$ be a map such that $f$ is Cauchy at $x_0$ with special limit $a$. If $f$ is non-expansive, then $G(\langle a\rangle_1^{n-1}, fa) = G(\langle a\rangle_1^n)$ and $G(\langle fa\rangle_1^{n-1}, a) \le G(\langle a\rangle_1^n)$.
\end{lem}
\begin{proof}
By condition {\bf(ssd)}, we have that $G(\langle a\rangle_1^n) \le G(\langle a\rangle_1^{n-1}, fa)$. Conversely, for any fixed $m\in\bN$, we have by condition {\bf(ptri)}
\begin{eqnarray*}
G(\langle a\rangle_1^{n-1}, fa) & \le & G(\langle a\rangle_1^{n-1}, f^{m+1}x_0) + G(\langle f^{m+1}x_0\rangle_1^{n-1}, fa) - G(\langle f^{m+1}x_0\rangle_1^n) \\
& \le & G(\langle a\rangle_1^{n-1}, f^{m+1}x_0) + G(\langle f^mx_0\rangle_1^{n-1}, a) - G(\langle f^{m+1}x_0\rangle_1^n).
\end{eqnarray*}
The last step is due to the non-expansiveness of $f$. Taking the limit as $m\to\infty$ and using Lemma~\ref{speciallimitprop}, we obtain
\begin{eqnarray*}
G(\langle a\rangle_1^{n-1}, fa) & \le & \lim_{m\to\infty}G(\langle a\rangle_1^{n-1}, f^{m+1}x_0) + \lim_{m\to\infty}G(\langle f^mx_0\rangle_1^{n-1}, a) - \lim_{m\to\infty}G(\langle f^{m+1}x_0\rangle_1^n) \\
& = & G(\langle a\rangle_1^n) + G(\langle a\rangle_1^n) - G(\langle a\rangle_1^n) \\
& = & G(\langle a\rangle_1^n).
\end{eqnarray*}
Consequently, we have $G(\langle a\rangle_1^{n-1}, fa) = G(\langle a\rangle_1^n)$.

On the other hand, for every fixed $m\in\bN$, we also have by condition {\bf(ptri)} again
\begin{eqnarray*}
G(\langle fa\rangle_1^{n-1}, a) & \le & G(\langle fa\rangle_1^{n-1}, f^{m+1}x_0) + G(\langle f^{m+1}x_0\rangle_1^{n-1}, a) - G(\langle f^{m+1}x_0\rangle_1^n) \\
& \le & G(\langle a\rangle_1^{n-1}, f^mx_0) + G(\langle f^{m+1}x_0\rangle_1^{n-1}, a) - G(\langle f^{m+1}x_0\rangle_1^n).
\end{eqnarray*}
Taking the limit as $m\to\infty$ and using Lemma~\ref{speciallimitprop} again, we obtain that
\begin{eqnarray*}
G(\langle fa\rangle_1^{n-1}, a) & \le & \lim_{m\to\infty}G(\langle a\rangle_1^{n-1}, f^mx_0) + \lim_{m\to\infty}G(\langle f^{m+1}x_0\rangle_1^{n-1}, a) - \lim_{m\to\infty}G(\langle f^{m+1}x_0\rangle_1^n) \\
& = & G(\langle a\rangle_1^n) + G(\langle a\rangle_1^n) - G(\langle a\rangle_1^n) \\
& = & G(\langle a\rangle_1^n).
\end{eqnarray*}
And thus, we have $G(\langle fa\rangle_1^{n-1}, a) \le G(\langle a\rangle_1^n)$. 
\end{proof}

\begin{lem}		\label{orbcontcriterion}
Let $(X, G)$ be a partial $n$-metric space, $x_0\in X$ be an element, and $f:X\to X$ be a map such that $f$ is Cauchy at $x_0$ with special limit $a$. If $f$ is orbitally continuous at $x_0$ for $a$, then $G(\langle fa\rangle_1^{n-1}, a) = G(\langle fa\rangle_1^n)$ and $G(\langle a\rangle_1^{n-1}, fa) \le G(\langle fa\rangle_1^n)$. 
\end{lem}
\begin{proof}
Observe that since $f$ is orbitally continuous at $x_0$ for $a$, we have that $fa$ is a limit of $\langle f^mx_0\rangle_{m\in\bN}$ (and not necessarily a special limit).

By condition {\bf(ssd)}, we have that $G(\langle fa\rangle_1^n) \le G(\langle fa\rangle_1^{n-1}, a)$. Conversely, for any fixed $m\in\bN$, we have by condition {\bf(ptri)}
\begin{eqnarray*}
G(\langle fa\rangle_1^{n-1}, a) & \le & G(\langle fa\rangle_1^{n-1}, f^mx_0) + G(\langle f^mx_0\rangle_1^{n-1}, a) - G(\langle f^mx_0\rangle_1^n).
\end{eqnarray*}
Taking the limit as $m\to\infty$ and using Lemma~\ref{speciallimitprop}, we obtain that
\begin{eqnarray*}
G(\langle fa\rangle_1^{n-1}, a) & \le & \lim_{m\to\infty}G(\langle fa\rangle_1^{n-1}, f^mx_0) + \lim_{m\to\infty}G(\langle f^mx_0\rangle_1^{n-1}, a) - \lim_{m\to\infty}G(\langle f^mx_0\rangle_1^n) \\
& = & G(\langle fa\rangle_1^n) + G(\langle a\rangle_1^n) - G(\langle a\rangle_1^n) \\
& = & G(\langle fa\rangle_1^n).
\end{eqnarray*}
Consequently, we have $G(\langle fa\rangle_1^{n-1}, a) = G(\langle fa\rangle_1^n)$.

On the other hand, for any fixed $m\in\bN$, we have by condition {\bf(ptri)} and Lemma~\ref{basicinequality}(c),
\begin{eqnarray*}
G(\langle a\rangle_1^{n-1}, fa) & \le & G(\langle a\rangle_1^{n-1}, f^mx_0) + G(\langle f^mx_0\rangle_1^{n-1}, fa) - G(\langle f^mx_0\rangle_1^n) \\
& \le & G(\langle a\rangle_1^{n-1}, f^mx_0) + (n-1)\, G(\langle fa\rangle_1^{n-1}, f^mx_0) - (n-2)\,G(\langle fa\rangle_1^n) \\
& & -\, G(\langle f^mx_0\rangle_1^n).
\end{eqnarray*}
Taking the limit as $m\to\infty$ and using Lemma~\ref{speciallimitprop} again, we obtain that
\begin{eqnarray*}
G(\langle a\rangle_1^{n-1}, fa) & \le & \lim_{m\to\infty}G(\langle a\rangle_1^{n-1}, f^mx_0) + (n-1)\lim_{m\to\infty} G(\langle fa\rangle_1^{n-1}, f^mx_0) \\
& & -\, \lim_{m\to\infty}G(\langle f^mx_0\rangle_1^n) -  (n-2)\,G(\langle fa\rangle_1^n)\\
& = & G(\langle a\rangle_1^n) + (n-1)\,G(\langle fa\rangle_1^n) - G(\langle a\rangle_1^n) - (n-2)\,G(\langle fa\rangle_1^n) \\
& = & G(\langle fa\rangle_1^n).
\end{eqnarray*}
And thus, we have $G(\langle a\rangle_1^{n-1}, fa) \le G(\langle fa\rangle_1^n)$.
\end{proof}

We are now ready to prove our main theorems.
\newline\newline
\underline{\bf Proof of Theorem~\ref{cauchyorbcontnexpfpt}}
\begin{proof}
We deal with the three cases separately.
\newline
{\bf Case I:} $f$ is non-expansive and orbitally continuous at $x_0$ for $a$.
\newline
Since $f$ is non-expansive, it follows by Lemma~\ref{nonexpcriterion} that $G(\langle a\rangle_1^{n-1}, fa) = G(\langle a\rangle_1^n)$.\newline
Since $f$ is orbitally continuous at $x_0$ for $a$, it follows by Lemma~\ref{orbcontcriterion} that 
$$G(\langle fa\rangle_1^{n-1}, a) = G(\langle fa\rangle_1^n).$$
Thus, by condition {\bf(sep)}, we have that $fa = a$, i.e., $a$ is a fixed point of $f$.
\newline\newline
{\bf Case II:} $f$ is orbitally continuous at $x_0$ for $a$ and the self distances of $(X, G)$ are bounded below by $G(\langle fa\rangle_1^n)$.
\newline
Since $f$ is orbitally continuous at $x_0$ for $a$, it follows by Lemma~\ref{orbcontcriterion} that 
$$G(\langle fa\rangle_1^{n-1}, a) = G(\langle fa\rangle_1^n) \;\;\mbox{ and }\;\; G(\langle a\rangle_1^{n-1}, fa) \le G(\langle fa\rangle_1^n).$$
Since the self distances of $(X, G)$ are bounded below by $G(\langle fa\rangle_1^n)$, it follows that 
$$G(\langle fa\rangle_1^n) \le G(\langle a\rangle_1^n).$$
By condition {\bf(ssd)}, we have that $G(\langle a\rangle_1^n) \le G(\langle a\rangle_1^{n-1}, fa)$. \newline
Combining all of these together, we get that $G(\langle a\rangle_1^{n-1}, fa) = G(\langle a\rangle_1^n)$. \newline
Consequently, by condition {\bf(sep)}, we have that $fa = a$, i.e., $a$ is a fixed point of $f$.
\newline\newline
{\bf Case III:} $f$ is non-expansive and the self distances of $(X, G)$ are bounded below by $G(\langle a\rangle_1^n)$.
\newline
Since $f$ is non-expansive, it follows by Lemma~\ref{nonexpcriterion} that 
$$G(\langle a\rangle_1^{n-1}, fa) = G(\langle a\rangle_1^n) \;\;\mbox{ and }\;\; G(\langle fa\rangle_1^{n-1}, a) \le G(\langle a\rangle_1^n).$$ 
Since the self distances of $(X, G)$ are bounded below by $G(\langle a\rangle_1^n$, it follows that 
$$G(\langle a\rangle_1^n) \le G(\langle fa\rangle_1^n).$$
By condition {\bf(ssd)}, we have that $G(\langle fa\rangle_1^n) \le G(\langle fa\rangle_1^{n-1}, a)$. \newline 
Combining all of these together, we get that $G(\langle fa\rangle_1^{n-1}, a) = G(\langle fa\rangle_1^n)$. \newline
Hence, by condition {\bf(sep)}, we have that $fa = a$, i.e., $a$ is a fixed point of $f$.
\end{proof}

\vspace{1em}
\hspace{-1.0em}\underline{\bf Proof of Theorem~\ref{strongcauchyorbcontnexpfpt}}
\begin{proof}
We deal with the two cases separately.
\newline
{\bf Case I:} $f$ is non-expansive.
\newline
Since $f$ is non-expansive, it follows by Lemma~\ref{nonexpcriterion} that 
$$G(\langle a\rangle_1^{n-1}, fa) = G(\langle a\rangle_1^n).$$
Since $(X, G)$ is strong, it then follows by condition ({\bf sssd}) that $fa = a$.
\newline\newline
{\bf Case II:} $f$ is orbitally continuous at $x_0$ for $a$.
\newline
Since $f$ is orbitally continuous at $x_0$ for $a$, it follows by Lemma~\ref{orbcontcriterion} that 
$$G(\langle fa\rangle_1^{n-1}, a) = G(\langle fa\rangle_1^n).$$
Since $(X, G)$ is strong, it then follows by condition ({\bf sssd}) that $fa = a$.
\end{proof}

\section{Orbitally $r$-contractive maps}
Let $(X, G)$ be a partial $n$-metric space. In the previous section, we showed the existence of a fixed point for a function $f:X\to X$ under the assumption that there is an element $x_0\in X$ such that $f$ is Cauchy at $x_0$. In this section, we give an example of a particular class of functions, which we call ``orbitally $r$-contractive'', that in fact satisfies this condition. Lemma~\ref{orbcontrCauchy} establishes this claim. These functions are our analogues of contractive (rather, orbitally contractive) functions suitable to our context.
\begin{defn}		\label{selfcontmap}
Take a partial $n$-metric space $(X, G)$, an element $x_0\in X$, a number $r\in\bR$, and a map $f: X\to X$. We say $f$ is {\em orbitally $r$-contractive at $x_0$} if there exists a real number $c$ with $0\le c < 1$ such that the following two conditions hold for all $m\in\bN$:
\begin{itemize}
\item $r\le G(\langle f^mx_0\rangle_1^n)$
\item $G(\langle f^mx_0\rangle_1^{n-1}, f^{m+1}x_0) \le r + c^m\,|G(\langle x_0\rangle_1^{n-1}, fx_0)|$.
\end{itemize}
And we say $f$ is {\em orbitally $r$-contractive} if $f$ is orbitally $r$-contractive at every $x\in X$.
\end{defn}

\begin{lem}		\label{orbcontrCauchy}
For each partial $n$-metric space $(X, G)$, element $x_0\in X$, real number $r\in\bR$, and map $f:X\to X$ orbitally $r$-contractive at $x_0$, the orbit $\langle f^mx_0\rangle_{m\in\bN}$ of $x_0$ under $f$ is a Cauchy sequence in $(X, G)$ with $\lim_{m_1, \ldots, m_n\to\infty}G(\langle f^{m_i}x_0\rangle_{i=1}^n) = r$.
\end{lem}
\begin{proof}
Since $f$ is orbitally $r$-contractive at $x_0$, there is $0\le c < 1$ such that for all $m\in\bN$
\begin{eqnarray*}
& & r\le G(\langle f^mx_0\rangle_1^n) \\
& & G(\langle f^mx_0\rangle_1^{n-1}, f^{m+1}x_0) \le r + c^m\,|G(\langle x_0\rangle_1^{n-1}, fx_0)|.
\end{eqnarray*}
Let $m_1, m_2$ be arbitrary. Without loss of generality, we can assume $m_2 > m_1$ (a similar argument works if $m_1 > m_2$). Write $m_2 = m_1+k+1$ for some $k \ge 0$. Then we have
\begin{eqnarray*}		\label{eqn1}
\nonumber
& & G(\langle f^{m_1}x_0\rangle_1^{n-1}, f^{m_2}x_0) \\
& = & G(\langle f^{m_1}x_0\rangle_1^{n-1}, f^{m_1+k+1}x_0) \\
& \le & G(\langle f^{m_1}x_0\rangle_1^{n-1}, f^{m_1+1}x_0) + G(\langle f^{m_1+1}x_0\rangle_1^{n-1}, f^{m_1+k+1}x_0) - G(\langle f^{m_1+1}x_0\rangle_1^n) \\
& \le & r + c^{m_1}\,|G(\langle x_0\rangle_1^{n-1}, fx_0)| + G(\langle f^{m_1+1}x_0\rangle_1^{n-1}, f^{m_1+k+1}x_0) - r \\
& = & c^{m_1}\,|G(\langle x_0\rangle_1^{n-1}, fx_0)| + G(\langle f^{m_1+1}x_0\rangle_1^{n-1}, f^{m_1+k+1}x_0) \\
& \le & \cdots \\
& \le & (c^{m_1} + \cdots + c^{m_1+k-1})\,|G(\langle x_0\rangle_1^{n-1}, fx_0)| + G(\langle f^{m_1+k}x_0\rangle_1^{n-1}, f^{m_1+k+1}x_0)
\end{eqnarray*}
\begin{eqnarray*}
& \le & (c^{m_1} + \cdots + c^{m_1+k-1})\,|G(\langle x_0\rangle_1^{n-1}, fx_0)| + r + c^{m_1+k}|G(\langle x_0\rangle_1^{n-1}, fx_0)| \\
& = & (c^{m_1} + \cdots + c^{m_1+k})\,|G(\langle x_0\rangle_1^{n-1}, fx_0)| + r \\
& = & r + c^{m_1}\dfrac{1 - c^{k+1}}{1 - c}\,|G(\langle x_0\rangle_1^{n-1}, fx_0)| \\
& \le & r + c^{m_1}\dfrac{1}{1 - c}\,|G(\langle x_0\rangle_1^{n-1}, fx_0)|.
\end{eqnarray*}
Taking the limit as $m_1\to\infty$, the right hand side of the above inequality goes to $r$ (since $0 \le c < 1$). Since also $r\le G(\langle f^{m_1}x_0\rangle_1^n)\le G(\langle f^{m_1}x_0\rangle_1^{n-1}, f^{m_2}x_0)$ for all $m_1, m_2$, we have
$$\lim_{m_1, m_2\to\infty} G(\langle f^{m_1}x_0\rangle_1^{n-1}, f^{m_2}x_0) = r.$$
By Lemma~\ref{equivCauchy}, we thus obtain
$$\lim_{m_1, \ldots, m_n\to\infty} G(\langle f^{m_i}x_0\rangle_{i=1}^n) = r,$$
and hence $\langle f^mx_0\rangle_{m\in\bN}$ is Cauchy.
\end{proof}

It thus follows that an orbitally $r$-contractive map is Cauchy for every $r\in\bR$. Because of this property, the orbitally $r$-contractive functions provide more examples of fixed point theorems. But for that we need the existence of special limits. The following weakening of completeness suffices for our fixed point theorems to work. 
\begin{defn}		\label{orbcomp}
Given a partial $n$-metric space $(X, G)$ and a map $f:X\to X$, the space $(X, G)$ is called {\em orbitally complete for $f$} if every Cauchy sequence in $(X, G)$ of the form $\langle f^mx_0\rangle_{m\in\bN}$, for $x_0\in X$, has a special limit $a\in X$.
\end{defn}

Combining this with the results of the previous section, we obtain the following.
\begin{cor}		\label{cauchyorbcontractivefpt}
Let $(X, G)$ be a partial $n$-metric space, $r\in\bR$, $x_0\in X$ and $f:X\to X$ be a map such that $f$ is orbitally $r$-contractive at $x_0$ and $(X, G)$ is orbitally complete for $f$. Further assume that one of the following holds:
\begin{enumerate}
\item $f$ is non-expansive and orbitally continuous at $x_0$;
\item $f$ is non-expansive and the self distances of $(X, G)$ are bounded below by $r$.
\end{enumerate}
Then there exists $a\in X$ such that $fa = a$ and $G(\langle a\rangle_1^n) = r$. 
\end{cor}
\begin{proof}
By Lemma~\ref{orbcontrCauchy}, the orbit $\langle f^mx_0\rangle_{m\in\bN}$ of $x_0$ under $f$ is a Cauchy sequence with $\lim_{m_1, \ldots, m_n\to\infty} G(\langle f^{m_i}x_0\rangle_{i=1}^n) = r$. Since $(X, G)$ is orbitally complete for $f$, there is an element $a\in X$ such that $a$ is a special limit of $\langle f^mx_0\rangle_{m\in\bN}$. By Definition~\ref{cauchylimit}, we have
$$G(\langle a\rangle_1^n) = \lim_{m_1, \ldots, m_n\to\infty} G(\langle f^{m_i}x_0\rangle_{i=1}^n) = r.$$
Finally, by Theorem~\ref{cauchyorbcontnexpfpt}, we have $fa = a$, i.e., $a$ is a fixed point of $f$.
\end{proof}

An analogous proof using Theorem~\ref{strongcauchyorbcontnexpfpt} instead of Theorem~\ref{cauchyorbcontnexpfpt} then gives the following.
\begin{cor}		\label{strongcauchyorbcontractivefpt}
Let $(X, G)$ be a strong partial $n$-metric space, $r\in\bR$, $x_0\in X$ and $f:X\to X$ be a map such that $f$ is orbitally $r$-contractive at $x_0$ and $(X, G)$ is orbitally complete for $f$. Further assume that one of the following holds:
\begin{enumerate}
\item $f$ is non-expansive;
\item $f$ is orbitally continuous at $x_0$.
\end{enumerate}
Then there exists $a\in X$ such that $fa = a$ and $G(\langle a\rangle_1^n) = r$.
\end{cor}

\section{Orbitally $\phi_r$-contractive maps}
In this section, we define another class of functions, which we call ``orbitally $\phi_r$-contractive'', that also satisfies the property of being Cauchy and for which we get similar fixed point theorems. Lemma~\ref{phircontmaplem} establishes this claim. This definition generalizes the corresponding definition used in \cite{SAKP_PM} for partial metric spaces.

\begin{defn}	\label{phircont}
Take a partial $n$-metric space $(X, G)$, an element $x_0\in X$, a number $r\in\bR$, and a map $f:X\to X$. We say $f$ is {\em orbitally $\phi_r-$contractive at $x_0$} if there exists a continuous non-decreasing function $\phi: [r, \infty)\to [0, \infty)$ with $\phi(r) = 0$ and $\phi(t) > 0$ for all $t > r$ such that the following two conditions hold for all $m_1, m_2\in\bN$
\begin{itemize}
\item $r \le G(\langle f^{m_1}x_0\rangle_1^n)$
\item $G(\langle f^{m_1 + 1}x_0\rangle_1^{n-1}, f^{m_2 + 1}(x)) \le G(\langle f^{m_1}x_0\rangle_1^{n-1}, f^{m_2}x_0) - \phi(G(\langle f^{m_1}x_0\rangle_1^{n-1}, f^{m_2}x_0))$.
\end{itemize}
We say $f$ is {\em orbitally $\phi_r-$contractive} if it is orbitally $\phi_r-$contractive at every $x\in X$.
\end{defn}

\begin{lem}		\label{phircontmaplem}
For each partial $n$-metric space $(X, G)$, element $x_0\in X$, real number $r\in\bR$, and map $f:X\to X$ orbitally $\phi_r$-contractive at $x_0$, the orbit $\langle f^mx_0\rangle_{m\in\bN}$ of $x_0$ under $f$ is a Cauchy sequence in $(X, G)$ with $\lim_{m_1, \ldots, m_n\to\infty} G(\langle f^{m_i}x_0\rangle_{i=1}^n) = r$.
\end{lem}
\begin{proof}
Set $x_{m+1} := fx_m$ for $m\in\bN$. \newline
Let $\phi: [r, \infty)\to [0, \infty)$ witness the fact that $f$ is orbitally $\phi_r$-contractive at $x_0$. In particular, $\phi$ is a continuous non-decreasing function with $\phi(r) = 0$ and $\phi(t) > 0$ for all $t > r$. Then, for all $m\in\bN$, we have
$$r\le G(\langle x_{m+2}\rangle_1^n) \le G(\langle x_{m+2}\rangle_1^{n-1}, x_{m+1}) \le G(\langle x_{m+1}\rangle_1^{n-1}, x_{m}) - \phi(G(\langle x_{m+1}\rangle_1^{n-1}, x_{m})).$$
Set $t_m := G(\langle x_{m+1}\rangle_1^{n-1}, x_{m})$. Then one obtains
\begin{eqnarray}		\label{tnseq}
r\le t_{m+1}\le t_m - \phi(t_m) \le t_m.
\end{eqnarray}
This implies that $\langle t_m\rangle_{m\in\bN}$ is a non-increasing sequence of real numbers bounded below by $r$, and hence converges to some $L\ge r$. We claim that $L = r:$ otherwise $L > r$, and hence $\phi(L) > 0$. Since $\phi$ is non-decreasing, we get $\phi(L) \le \phi(t_m)$ for all $m\in\bN$. Due to (\ref{tnseq}), we have $t_{m+1} \le t_m - \phi(t_m) \le t_m - \phi(L)$, and so
$$t_{m+2} \le t_{m+1} - \phi(t_{m+1}) \le t_m - \phi(t_m) - \phi(t_{m+1}) \le t_m - 2\phi(L).$$
Inductively, we obtain $t_{m+k} \le t_m - k\phi(L)$, which is a contradiction for large enough $k\in\bN$. Thus, we have $\phi(L) = 0$, and hence $L = r$. Consequently, $\lim_{m\to\infty} G(\langle x_{m+1}\rangle_1^{n-1}, x_{m}) = r$.

Now we show that
$$\lim_{m_1, m_2\to\infty} G(\langle x_{m_1}\rangle_1^{n-1}, x_{m_2}) = r.$$
Suppose it is not the case. Then there exists $\epsilon_0 > r$ and two sequences of integers $\langle m_1(k)\rangle_{k\in\bN}$ and $\langle m_2(k)\rangle_{k\in\bN}$ such that $m_1(k) > m_2(k) \ge k$ and
\begin{eqnarray}		\label{greatereps}
s_k := G(\langle x_{m_1(k)}\rangle_1^{n-1}, x_{m_2(k)}) \ge\epsilon_0
\end{eqnarray}
for all $k\in\bN$. Since $\lim_{m\to\infty} G(\langle x_{m+1}\rangle_1^{n-1}, x_{m}) = r$, we can also assume without loss of generality that $G(\langle x_{m_1(k)-1}\rangle_1^{n-1}, x_{m_2(k)}) < \epsilon_0$ for each $k\in\bN$. Thus, we have
\begin{eqnarray*}
\epsilon_0 \;\le\; s_k & = & G(\langle x_{m_1(k)}\rangle_1^{n-1}, x_{m_2(k)}) \\
& \le & G(\langle x_{m_1(k)}\rangle_1^{n-1}, x_{m_1(k)-1}) + G(\langle x_{m_1(k) - 1}\rangle_1^{n-1}, x_{m_2(k)}) - G(\langle x_{m_1(k) - 1}\rangle_1^n) \\
& < & t_{m_1(k) - 1} + \epsilon_0 - r \\
& \le & t_k + \epsilon_0 - r.
\end{eqnarray*}
Since $\lim_{k\to\infty} t_k = r$, we have $\lim_{k\to\infty} (t_k + \epsilon_0 - r) = r + \epsilon_0 - r = \epsilon_0$. Consequently,
$$\lim_{k\to\infty} s_k = \epsilon_0.$$ 
On the other hand, by Lemma~\ref{basicinequality}(b), we have
\begin{eqnarray*}
s_k & = & G(\langle x_{m_1(k)}\rangle_1^{n-1}, x_{m_2(k)}) \\
& \le & G(\langle x_{m_1(k)+1}\rangle_1^{n-1}, x_{m_2(k)+1}) + (G(\langle x_{m_2(k)+1}\rangle_1^{n-1}, x_{m_2(k)}) - G(\langle x_{m_2(k)+1}\rangle_1^n)) \\
& & +\; \sum_{j=1}^{n-1} (G(\langle x_{m_1(k)+1}\rangle_1^{n-1}, x_{m_1(k)}) - G(\langle x_{m_1(k)+1}\rangle_1^n)) \\
& \le & (n-1)\, t_{m_1(k)} + t_{m_2(k)} - nr + G(\langle x_{m_1(k)+1}\rangle_1^{n-1}, x_{m_2(k)+1}) \\
& \le & nt_k - nr + G(\langle x_{m_1(k)}\rangle_1^{n-1}, x_{m_2(k)}) - \phi(G(\langle x_{m_1(k)}\rangle_1^{n-1}, x_{m_2(k)}))\\
& = & nt_k - nr + s_k - \phi(s_k) \\
\implies \phi(s_k) & \le & nt_k - nr.
\end{eqnarray*}
Again, since $\lim_{k\to\infty} t_k = r$, we have $\lim_{k\to\infty} (nt_k - nr) = nr - nr = 0$. Since $\phi(s_k) \ge 0$ for all $k\in\bN$ (by the definition of $\phi$), we get that $\lim_{k\to\infty}\phi(s_k) = 0$. Since $\phi$ is continuous, it then follows that
$$0 = \lim_{k\to\infty}\phi(s_k) = \phi\Big(\lim_{k\to\infty} s_k\Big) = \phi(\epsilon_0),$$
which contradicts the fact that $\epsilon_0 > r$. Hence, $\lim_{m_1, m_2\to\infty} G(\langle x_{m_1}\rangle_1^{n-1}, x_{m_2}) = r$.\newline
By Lemma~\ref{equivCauchy}, we thus obtain
$$\lim_{m_1, \ldots, m_n\to\infty} G(\langle x_{m_i}\rangle_{i=1}^n) = r,$$
and hence $\langle f^mx_0\rangle_{m\in\bN}$ is Cauchy.
%
\end{proof}

It thus follows that an orbitally $\phi_r$-contractive map is Cauchy for every $r\in\bR$. Because of the Cauchy property, the orbitally $\phi_r$-contractive functions provide further examples of fixed point theorems. By similar proofs as we had for Corollaries~\ref{cauchyorbcontractivefpt} \& \ref{strongcauchyorbcontractivefpt}, we obtain
\begin{cor}		\label{phircontmapcor}
Let $(X, G)$ be a partial $n$-metric space, $r\in\bR$, $x_0\in X$, and $f:X\to X$ be a map such that $f$ is orbitally $\phi_r$-contractive at $x_0$ and $(X, G)$ is orbitally complete for $f$. Further assume that one of the following holds:
\begin{enumerate}
\item $f$ is non-expansive and orbitally continuous at $x_0$; 
\item $f$ is non-expansive and the self distances of $(X, G)$ are bounded below by $r$.
\end{enumerate}
Then there exists $a\in X$ such that $fa = a$ and $G(\langle a\rangle_1^n) = r$.
\end{cor}

\begin{cor}		\label{phircontmapcorstrong}
Let $(X, G)$ be a strong partial $n$-metric space, $r\in\bR$, $x_0\in X$, and $f:X\to X$ be a map such that $f$ is orbitally $\phi_r$-contractive at $x_0$ and $(X, G)$ is orbitally complete for $f$. Further assume that one of the following holds:
\begin{enumerate}
\item $f$ is non-expansive;
\item $f$ is orbitally continuous at $x_0$.
\end{enumerate}
Then there exists $a\in X$ such that $fa = a$ and $G(\langle a\rangle_1^n) = r$.
\end{cor}

\section{$n$-metric space}
Finally, in this last section, we give a new definition of an $n$-metric space that generalizes the notion of a metric space. This definition is weaker than the definition of a generalized $n$-metric given by Khan \cite{KAK_nmetric}. In particular, our $n$-metric is not required to satisfy condition [G3] of Khan's definition.
\begin{defn}
Let $(X, G)$ be a partial $n$-metric space. We say $(X, G)$ is an {\em $n$-metric space} if for all $x\in X$, we have
$$G(\langle x\rangle_1^n) = 0.$$
\end{defn}

We have the following result.
\begin{lem}		\label{nmetricT2}
Let $(X, G)$ be an $n$-metric space. Then for all $x, y\in X$ with $x\neq y$, we have
$$G(\langle x\rangle_1^{n-1}, y) > 0.$$
\end{lem}
\begin{proof}
Let $x, y\in X$ with $x\neq y$.

By condition {\bf(sep)}, we have $G(\langle x\rangle_1^{n-1}, y) \ge G(\langle x\rangle_1^n) = 0$.

For a contradiction, let us assume that $G(\langle x\rangle_1^{n-1}, y) = 0$, i.e., $G(\langle x\rangle_1^{n-1}, y) = G(\langle x\rangle_1^n)$.

By Lemma~\ref{basicinequality}(c), we obtain
$$0\le G(\langle y\rangle_1^{n-1}, x) \le (n-1)\,G(\langle x\rangle_1^{n-1}, y) - (n-2)\,G(\langle x\rangle_1^n) = 0,$$
i.e., $G(\langle y\rangle_1^{n-1}, x) = 0 = G(\langle y\rangle_1^n)$.

By condition {\bf(sep)}, it follows that $x = y$, which gives our required contradiction.
\end{proof}

As before (cf. Lemma~\ref{partialnmetricdistance}), one can define a metric $d_G:X\times X\to\bR^{\ge 0}$ on an $n$-metric space $(X, G)$ as follows:
$$d_G(x, y) := G(\langle x\rangle_1^{n-1}, y) + G(\langle y\rangle_1^{n-1}, x).$$
Let us denote the corresponding metric topology induced by $d_G$ on $X$ as $\tau[d_G]$. Then we have the following result.
\begin{prop}		\label{nmetricsameasmetric}
Let $(X, G)$ be an $n$-metric space. Then the topologies $\tau[G]$ and $\tau[d_G]$ are the same.
\end{prop}
\begin{proof}
Fix $x\in X$ and $\epsilon > 0$. Let us denote the open ball with center $x$ and radius $\epsilon$ in the topology $\tau[G]$ by $B^G_\epsilon(x)$ and that in the topology $\tau[d_G]$ by $B^{d_G}_\epsilon(x)$, i.e.,
\begin{eqnarray*}
B^G_\epsilon(x) & = & \{y\in X\mid G(\langle x\rangle_1^{n-1}, y) < \epsilon\} \\
B^{d_G}_\epsilon(x) & = & \{y\in X\mid d_G(x, y) < \epsilon\}.
\end{eqnarray*}
We now show that
$$B^G_{\frac{\epsilon}{n}}(x) \subseteq B^{d_G}_\epsilon(x) \subseteq B^G_\epsilon(x),$$
which suffices to show that the two topologies $\tau[G]$ and $\tau[d_G]$ are the same.

It follows trivially from the definition of $d_G$ and the fact that $G(\langle x\rangle_1^{n-1}, y)\ge 0$ for all $x, y\in X$ that
$$B^{d_G}_\epsilon(x) \subseteq B^G_\epsilon(x).$$
On the other hand, we have by Lemma~\ref{basicinequality}(c) that
$$d_G(x, y) = G(\langle x\rangle_1^{n-1}, y) + G(\langle y\rangle_1^{n-1}, x) \le G(\langle x\rangle_1^{n-1}, y) + (n-1) G(\langle x\rangle_1^{n-1}, y) = nG(\langle x\rangle_1^{n-1}, y).$$
And hence, it follows trivially that $B^G_{\frac{\epsilon}{n}}(x) \subseteq B^{d_G}_\epsilon(x)$.
\end{proof}

Now we give an example of a two-point 5-metric space to illustrate the fact that the value of $G$ can be negative on certain tuples even though all the self-distances are zero. This also illustrates the fact that our axioms for a partial $n$-metric do not require the images of all $n$-tuples under $G$ to be comparable, even in an $n$-metric space $(X, G)$. Finally, this also illustrates why our definition is weaker than that of Khan's.
\begin{example}
Let $X = \{a, b\}$ be a two-point space. Define $G:X^5\to\bR$ as follows:
\begin{center}
\begin{tabular}{lcl}$G(a, a, a, a, a) = 0$ & & $G(b, b, b, b, b) = 0$\\ $G(a, b, b, b, b) = 4$ & & $G(b, a, a, a, a) = 3$ \\ $G(a, a, b, b, b) = 2$ & & $G(b, b, a, a, a) = -1$.\end{tabular}\\
Extend $G$ to the other tuples using condition {\bf(sym)}.
\end{center}
Observe that $G(b, b, a, a, a) = -1 < 0$. \newline
We leave it to the reader to verify that $(X, G)$ is indeed a 5-metric space.
\end{example}

\bibliographystyle{amsplain}
\bibliography{references}

\providecommand{\bysame}{\leavevmode\hbox to3em{\hrulefill}\thinspace}
\providecommand{\MR}{\relax\ifhmode\unskip\space\fi MR }
\providecommand{\MRhref}[2]{%
  \href{http://www.ams.org/mathscinet-getitem?mr=#1}{#2}
}
\providecommand{\href}[2]{#2}
\begin{thebibliography}{10}

\bibitem{SAKP_PM}
S.~Assaf and K.~Pal, \emph{Partial metric spaces with negative distances and
  fixed point theorems}, Submitted.

\bibitem{BDAPBR}
B.~C. Dhage, A.~M. Pathan, and B.~E. Rhoades, \emph{A general existence
  principle for fixed point theorems in d-metric spaces}, Internat. J. Math. \&
  Math. Sci. \textbf{23} (2000), no.~7, 441--448.

\bibitem{Gahler1}
S.~G{\"a}hler, \emph{2-metrische r{\"a}ume und ihre topologische struktur},
  Math. Nachr. \textbf{26} (1963), 115--148.

\bibitem{Gahler2}
\bysame, \emph{Zur geometric 2-metrische r{\"a}ume}, Rev. Roum. Math. Pures et
  Appl. \textbf{11} (1966), 664--669.

\bibitem{KSHYJCAW}
K.~S. Ha, Y.~J. Cho, and A.~White, \emph{Strictly convex and 2-convex 2-normed
  spaces}, Math. Japonica \textbf{33} (1988), 375--384.

\bibitem{KAK_nTop}
K.~A. Khan, \emph{On the possibility of n-topological spaces}, International
  Journal of Mathematical Archive \textbf{3} (2012), no.~7, 2520--2523.

\bibitem{KAK_nmetric}
\bysame, \emph{Generalized $n$-metric spaces and fixed point theorems}, Journal
  of Nonlinear and Convex Analysis \textbf{15} (2014), no.~6, 1221--1229.

\bibitem{SM1}
S.~G. Matthews, \emph{Partial metric topology}, Research Report 212, Dept. of
  Computer Science, University of Warwick, 1992.

\bibitem{SM2}
\bysame, \emph{Partial metric topology}, In Proc. 8th Summer Conference on
  General Topology and Applications, Annals of the New York Academy of Sciences
  \textbf{728} (1994), 183--197.

\bibitem{ZMBS_rem}
Z.~Mustafa and B.~Sims, \emph{Some remarks concerning d-metric spaces},
  Proceedings of the International Conferences on Fixed Point Theory and
  Applications, Valencia (Spain) (2003), 189--198.

\bibitem{ZMBS_GM}
\bysame, \emph{A new approach to generalized metric spaces}, Journal of
  Nonlinear Convex Analysis \textbf{7} (2006), no.~2, 289--297.

\bibitem{SJON}
S.~J. O'Neill, \emph{Partial metrics, valuations and domain theory}, In Proc.
  11th Summer Conference on General Topology and Applications, Annals of the
  New York Academy of Sciences \textbf{806} (1996), 304--315.

\bibitem{AZDN}
M.~R.~A. Zand and A.~D. Nezhad, \emph{A generalization of partial metric
  spaces}, Journal of Contemporary Applied Mathematics \textbf{1} (2011),
  no.~1, 86--93.

\end{thebibliography}

\end{document}